\documentclass[a4paper,11pt]{article}%
\usepackage{amsmath}
\usepackage{amsfonts}
\usepackage{amssymb}
\usepackage{graphicx}
\usepackage{hyperref}
\usepackage{textcomp}%
\providecommand{\U}[1]{\protect\rule{.1in}{.1in}}
\newtheorem{theorem}{Theorem}[section]

\newtheorem{corollary}[theorem]{Corollary}

\newtheorem{definitions}[theorem]{Definitions}

\newtheorem{lemma}[theorem]{Lemma}

\newtheorem{proposition}[theorem]{Proposition}

\numberwithin{equation}{section}
\newenvironment{proof}[1][Proof]{\noindent\textbf{#1.} }{\ \rule{0.5em}{0.5em}}
\hypersetup{
colorlinks   = true,   urlcolor     = blue,   linkcolor    = blue,   citecolor   = red }
\begin{document}

\title{Asymptotics for Sobolev extremals: the hyperdiffusive case}
\author{Grey Ercole\\{\small Universidade Federal de Minas Gerais} \\{\small Belo Horizonte, MG, 30.123-970, Brazil}\\{\small grey@mat.ufmg.br}}
\maketitle

\begin{abstract}
\noindent Let $\Omega$ be a bounded, smooth domain of $\mathbb{R}^{N},$
$N\geq2.$ For $p>N$ and $1\leq q(p)<\infty$ set
\[
\lambda_{p,q(p)}:=\inf\left\{  \int_{\Omega}\left\vert \nabla u\right\vert
^{p}\mathrm{d}x:u\in W_{0}^{1,p}(\Omega)\text{ \ and \ }\int_{\Omega
}\left\vert u\right\vert ^{q(p)}\mathrm{d}x=1\right\}
\]
and denote by $u_{p,q(p)}$ a positive extremal function corresponding to
$\lambda_{p,q(p)}$. We show that if $\lim\limits_{p\rightarrow\infty
}q(p)=\infty$, then $\lim\limits_{p\rightarrow\infty}\lambda_{p,q(p)}%
^{1/p}=\left\Vert d_{\Omega}\right\Vert _{\infty}^{-1}$, where $d_{\Omega}$
denotes the distance function to the boundary of $\Omega.$ Moreover, in the
hyperdiffusive case: $\lim\limits_{p\rightarrow\infty}\frac{q(p)}{p}=\infty,$
we prove that each sequence $u_{p_{n},q(p_{n})},$ with $p_{n}\rightarrow
\infty,$ admits a subsequence converging uniformly in $\overline{\Omega}$ to a
viscosity solution to the problem%
\[
\left\{
\begin{array}
[c]{lll}%
-\Delta_{\infty}u=0 & \text{in} & \Omega\setminus M\\
u=0 & \text{on} & \partial\Omega\\
u=1 & \text{in} & M,
\end{array}
\right.
\]
where $M$ is a closed subset of the set of all maximum points of $d_{\Omega}.$

\end{abstract}



\section{Introduction}

Let $\Omega$ be a smooth, bounded domain of $\mathbb{R}^{N},$ $N\geq2.$ For
$p>N$ and $1\leq q<\infty$ let
\begin{equation}
\lambda_{p,q}:=\inf\left\{  \left\Vert \nabla u\right\Vert _{p}^{p}:u\in
W_{0}^{1,p}(\Omega)\text{ \ and \ }\left\Vert u\right\Vert _{q}=1\right\}
\label{lpq}%
\end{equation}
be the best Sobolev constant of the embedding $W_{0}^{1,p}(\Omega
)\hookrightarrow L^{q}(\Omega).$ Here and in what follows $\left\Vert
\cdot\right\Vert _{r}$ denotes the standard norm of the Lebesgue space
$L^{r}(\Omega),$ $1\leq r\leq\infty,$ and $\left\Vert \nabla\cdot\right\Vert
_{p}$ denotes the standard norm of the Sobolev space $W_{0}^{1,p}(\Omega).$

As it is well known, $\lambda_{p,q}$ is in fact a minimum: there exists a
function $u_{p,q}\in W_{0}^{1,p}(\Omega)$ such that
\[
\left\Vert u_{p,q}\right\Vert _{q}=1\text{ \ and \ }\lambda_{p,q}=\left\Vert
\nabla u_{p,q}\right\Vert _{p}^{p}.
\]
Moreover, $u_{p,q}\in L^{\infty}(\Omega)$ and does not chance sign in
$\Omega.$ Throughout this paper, $u_{p,q}$ will denote a positive minimizer in
(\ref{lpq}). Such a function will be referred as a \textit{Sobolev extremal
function corresponding to }$\lambda_{p,q}$. As a consequence of its minimizing
property $u_{p,q}$ is a weak solution to the Dirichlet problem%
\begin{equation}
\left\{
\begin{array}
[c]{lll}%
-\Delta_{p}u=\lambda_{p,q}\left\vert u\right\vert ^{q-2}u & \text{\textup{in}}
& \Omega\\
u=0 & \text{\textup{on}} & \partial\Omega,
\end{array}
\right.  \label{Dir}%
\end{equation}
where $\Delta_{p}u:=\operatorname{div}\left(  \left\vert \nabla u\right\vert
^{p-2}\nabla u\right)  $ is the $p$-Laplacian operator.

The Sobolev extremal function $u_{p,q}$ is the only positive minimizer in
(\ref{lpq}) if $1\leq q\leq p,$ but in general this uniqueness property does
not hold if $p<q<\infty.$ For this matter we refer to the paper \cite{BLind}
by Brasco and Lindgren and references therein.

When $q\not =p$ the Dirichlet problem%
\begin{equation}
\left\{
\begin{array}
[c]{lll}%
-\Delta_{p}v=\lambda\left\vert v\right\vert ^{q-2}v & \text{\textup{in}} &
\Omega\\
v=0 & \text{\textup{on}} & \partial\Omega,
\end{array}
\right.  \label{dirq1}%
\end{equation}
admits at least one positive weak solution for each $\lambda>0.$

If $1\leq q<p$ the problem (\ref{dirq1}) has a unique positive solution
$v_{p,q}$ which is given by the expression
\begin{equation}
v_{p,q}=\left(  \frac{\lambda}{\lambda_{p,q}}\right)  ^{\frac{1}{p-q}}u_{p,q}.
\label{exp0}%
\end{equation}
Furthermore, $v_{p,q}$ minimizes globally (i.e. on $W_{0}^{1,p}(\Omega)$) the
\textit{energy functional}
\[
E_{\lambda}(v):=\frac{1}{p}\left\Vert \nabla v\right\Vert _{p}^{p}%
-\frac{\lambda}{q}\left\Vert v\right\Vert _{q}^{q}.
\]

If $q>p$ the problem (\ref{dirq1}) might, in principle, have multiple positive
weak solutions for each $\lambda>0.$ Among such solutions, those in which the
energy functional reaches the lowest value are known as \textit{least energy
solutions.} They are characterized as the positive minimizers of the energy
functional on the Nehari manifold%
\[
N_{\lambda}:=\left\{  v\in W_{0}^{1,p}(\Omega)\setminus\{0\}:\left\Vert \nabla
v\right\Vert _{p}^{p}=\lambda\left\Vert v\right\Vert _{q}^{q}\right\}  .
\]
An interesting fact is that least energy solutions can also be written in the
form (\ref{exp0}) for some Sobolev extremal function $u_{p,q}:$
\begin{equation}
v_{p,q}=\left(  \frac{\lambda_{p,q}}{\lambda}\right)  ^{\frac{1}{q-p}}u_{p,q}.
\label{exp}%
\end{equation}

When $q=p$ (\ref{dirq1}) becomes the homogeneous Dirichlet problem%
\begin{equation}
\left\{
\begin{array}
[c]{lll}%
-\Delta_{p}v=\lambda\left\vert v\right\vert ^{p-2}v & \text{\textup{in}} &
\Omega\\
v=0 & \text{\textup{on}} & \partial\Omega.
\end{array}
\right.  \label{dirp}%
\end{equation}
which is known as the eigenvalue problem of the Dirichlet $p$-Laplacian. In
this context, $\lambda_{p,p}$ is known as the first eigenvalue since it is the
least value of $\lambda$ for which (\ref{dirp}) has a nontrivial weak
solution. Moreover, (\ref{dirp}) has a positive weak solution $v_{p,p}$ if and
only if $\lambda=\lambda_{p,p}$ and $v_{p,p}=ku_{p,p},$ for some positive
constant $k.$

In \cite{JLM}, Juutinen, Lindqvist and Manfredi studied the asymptotic
behavior, as $p\rightarrow\infty,$ of the pair $(\lambda_{p,p},u_{p,p}).$ They
first deduced that the infimum
\begin{equation}
\Lambda_{\infty}:=\inf\left\{  \left\Vert \nabla u\right\Vert _{\infty}:u\in
C_{0}(\overline{\Omega})\cap W^{1,\infty}(\Omega)\text{ \ and \ }\left\Vert
u\right\Vert _{\infty}=1\right\}  \label{Linf}%
\end{equation}
is achieved by $d_{\Omega}/\left\Vert d_{\Omega}\right\Vert _{\infty},$ where
$d_{\Omega}$ denotes the distance function to the boundary of $\Omega$:%
\[
d_{\Omega}(x):=\inf_{y\in\partial\Omega}\left\vert x-y\right\vert ,\quad
x\in\overline{\Omega}.
\]
Thus,
\[
\Lambda_{\infty}=\left\Vert \nabla(d_{\Omega}/\left\Vert d_{\Omega}\right\Vert
_{\infty})\right\Vert _{\infty}=\left\Vert \nabla d_{\Omega}\right\Vert
_{\infty}/\left\Vert d_{\Omega}\right\Vert _{\infty}=\left\Vert d_{\Omega
}\right\Vert _{\infty}^{-1}.
\]
(We recall that $d_{\Omega}\in C_{0}(\overline{\Omega})\cap W^{1,\infty
}(\Omega)$ and $\left\vert \nabla d_{\Omega}\right\vert =1$ a.e. in $\Omega.$
Moreover, $\left\Vert d_{\Omega}\right\Vert _{\infty}$ is the inradius of
$\Omega,$ i.e. the radius of the greatest ball inscribed in $\Omega.$)

Still in \cite{JLM}, Juutinen, Lindqvist and Manfredi, showed that
\[
\lim_{p\rightarrow\infty}(\lambda_{p,p})^{1/p}=\Lambda_{\infty}%
\]
and proved that any sequence $u_{p_{n},p_{n}},$ with $p_{n}\rightarrow\infty,$
admits a subsequence converging uniformly in $\overline{\Omega}$ to a function
$u_{\infty}\in C_{0}(\overline{\Omega})\cap W^{1,\infty}(\Omega)$ which is a
positive minimizer in (\ref{Linf}), that is,
\[
u_{\infty}>0\text{ in }\Omega,\text{ \ }\left\Vert u_{\infty}\right\Vert
_{\infty}=1,\text{ \ and \ }\Lambda_{\infty}=\left\Vert \nabla u_{\infty
}\right\Vert _{\infty}.
\]
In addition, they proved that $u_{\infty}$ is a viscosity solution to the
Dirichlet problem%
\begin{equation}
\left\{
\begin{array}
[c]{lll}%
\min\left\{  \left\vert \nabla u\right\vert -\Lambda_{\infty}u,-\Delta
_{\infty}u\right\}  =0 & \mathrm{in} & \Omega\\
u=0 & \mathrm{on} & \partial\Omega,
\end{array}
\right.  \label{infeig}%
\end{equation}
where%
\[
\Delta_{\infty}u:=%
{\displaystyle\sum\limits_{i,j=1}^{N}}
u_{x_{i}}u_{x_{j}}u_{x_{i}x_{j}}%
\]
denotes the infinity Laplacian.

These results were independently obtained by Fukagai, Ito and Narukawa in
\cite{Fuka}, where the asymptotic behavior (as $p\rightarrow\infty$) of the
higher eigenvalues of the Dirichlet $p$-Laplacian was also studied.

The value $\Lambda_{\infty}=\left\Vert d_{\Omega}\right\Vert _{\infty}^{-1}$
has been referred in the literature as the first eigenvalue of the infinity
Laplacian and the Dirichlet problem (taken in the viscosity sense)
\[
\left\{
\begin{array}
[c]{lll}%
\min\left\{  \left\vert \nabla u\right\vert -\Lambda u,-\Delta_{\infty
}u\right\}  =0 & \mathrm{in} & \Omega\\
u=0 & \mathrm{on} & \partial\Omega
\end{array}
\right.
\]
has been referred as the eigenvalue problem for the infinity Laplacian. Thus,
the limit function $u_{\infty}$ obtained in \cite{Fuka} and \cite{JLM} is a
first eigenfunction of the infinity Laplacian.

Now, let us consider the Dirichlet problem
\begin{equation}
\left\{
\begin{array}
[c]{lll}%
-\Delta_{p}v=\mu_{p}\left\vert v\right\vert ^{q(p)-2}v & \mathrm{in} &
\Omega\\
v=0 & \mathrm{on} & \partial\Omega,
\end{array}
\right.  \label{dirq}%
\end{equation}
where $\mu_{p}>0$ is such that
\begin{equation}
\Lambda:=\lim_{p\rightarrow\infty}(\mu_{p})^{1/p}\in(0,\infty). \label{Lamb}%
\end{equation}

We will denote by $v_{p,q(p)}$ either the only positive weak solution to
(\ref{dirq}), if $1\leq q(p)<p,$ or any least energy solution to the same
problem, if $q(p)>p.$

Under (\ref{Lamb}) the asymptotic behavior of $v_{p,q(p)}$ was studied by
Charro and Peral in \cite{ChaPe} ($q(p)<p$) and Charro and Parini in
\cite{ChaPa} ($q(p)>p$). Each of these papers considered a different range for
the limit
\begin{equation}
Q:=\lim_{p\rightarrow\infty}\frac{q(p)}{p}. \label{Q}%
\end{equation}
In \cite{ChaPe}, Charro and Peral studied the \textit{subdiffusive} case:
$Q\in(0,1)$ whereas in \cite{ChaPa}, Charro and Parini studied the
\textit{superdiffusive} case: $Q\in(1,\infty)$.

In both works it is proved that any sequence $v_{p_{n},q(p_{n})},$ with
$p_{n}\rightarrow\infty,$ admits a subsequence converging uniformly to a
viscosity solution $v_{\infty}\in C_{0}(\overline{\Omega})\cap W^{1,\infty
}(\Omega)$ to the problem
\begin{equation}
\left\{
\begin{array}
[c]{lll}%
\min\left\{  \left\vert \nabla v\right\vert -\Lambda v^{Q},-\Delta_{\infty
}v\right\}  =0 & \mathrm{in} & \Omega\\
v=0 & \mathrm{on} & \partial\Omega.
\end{array}
\right.  \label{eQ}%
\end{equation}

Taking into account that, owing to (\ref{exp}),
\begin{equation}
v_{p,q(p)}=\left(  \frac{\lambda_{p,q(p)}}{\mu_{p}}\right)  ^{\frac{1}%
{q(p)-p}}u_{p,q(p)}, \label{scale}%
\end{equation}
the analysis of the behavior of $v_{p,q(p)}$ under (\ref{Lamb}) can be
restricted to the behavior of the pair $\left(  \lambda_{p,q(p)}%
,u_{p,q(p)}\right)  .$

In our first result, stated in the sequence, we simply assume that $q(p)$ goes
to infinity with $p.$ We do not impose any conditions on the sign of the
difference $q(p)-p$ or on the limit $Q$ defined in (\ref{Q}).

\begin{theorem}
\label{main1}If
\begin{equation}
\lim_{p\rightarrow\infty}q(p)=\infty, \label{qinf}%
\end{equation}
then%
\begin{equation}
\lim_{p\rightarrow\infty}\lambda_{p,q(p)}^{1/p}=\Lambda_{\infty}, \label{lim}%
\end{equation}
and
\begin{equation}
\lim_{p\rightarrow\infty}\left\Vert u_{p,q(p)}\right\Vert _{\infty}=1.
\label{lim1}%
\end{equation}

Moreover, each sequence $u_{p_{n},q(p_{n})},$ with $p_{n}\rightarrow\infty,$
admits a subsequence that converges uniformly to a function $u_{\infty}\in
C_{0}(\overline{\Omega})\cap W^{1,\infty}(\Omega)$ enjoying the following properties:

\begin{enumerate}
\item $0\leq u_{\infty}(x)\leq\Lambda_{\infty}d_{\Omega}(x)$ for all $x\in$
$\overline{\Omega}.$

\item $\left\Vert u_{\infty}\right\Vert _{\infty}=1$ \ and \ $\Lambda_{\infty
}=\left\Vert \nabla u_{\infty}\right\Vert _{\infty}.$

\item $M:=\left\{  x\in\Omega:u_{\infty}(x)=1\right\}  \subseteq M_{\Omega
}:=\left\{  x\in\Omega:d_{\Omega}(x)=\left\Vert d_{\Omega}\right\Vert
_{\infty}\right\}  .$

\item $u_{\infty}$ infinity superharmonic in $\Omega$ and (consequently)
positive in $\Omega.$
\end{enumerate}
\end{theorem}

Note from item 2 above that the function $u_{\infty}$ is also a solution to
the minimization problem given by (\ref{Linf}).

As far as we know, the limit (\ref{lim}) has not yet been observed in the
literature (except in the case $q(p)=p$ as described above). We remark from
(\ref{scale}) that if $Q\in(0,1)\cup(1,\infty]$ and $\mu_{p}$ satisfies
(\ref{Lamb}), then
\begin{equation}
\lim_{p\rightarrow\infty}v_{p,q(p)}=\lim_{p\rightarrow\infty}\left(
\frac{\lambda_{p,q(p)}}{\mu_{p}}\right)  ^{\frac{1}{q(p)-p}}u_{p,q(p)}=\left(
\frac{\Lambda_{\infty}}{\Lambda}\right)  ^{\frac{1}{Q-1}}\lim_{p\rightarrow
\infty}u_{p,q(p)} \label{exp1}%
\end{equation}
since the limit (\ref{lim}) yields%
\[
\lim_{p\rightarrow\infty}\left(  \frac{\lambda_{p,q(p)}}{\mu_{p}}\right)
^{\frac{1}{q(p)-p}}=\lim_{p\rightarrow\infty}\left(  \frac{\lambda_{p,q(p)}%
}{\mu_{p}}\right)  ^{\frac{1}{p}}{}^{\frac{1}{\frac{q(p)}{p}-1}}=\left(
\frac{\Lambda_{\infty}}{\Lambda}\right)  ^{\frac{1}{Q-1}}.
\]

Corollary \ref{cor1} below can be proved by adapting the proof of Theorem 1.21
in \cite{JLM} to the function $u_{\infty}$ obtained in Theorem \ref{main1}.

\begin{corollary}
\label{cor1}If
\[
0<Q<\infty,
\]
then each sequence $u_{p_{n},q(p_{n})},$ with $p_{n}\rightarrow\infty,$ admits
a subsequence that converges uniformly to viscosity solution to the Dirichlet
problem%
\[
\left\{
\begin{array}
[c]{lll}%
\min\left\{  \left\vert \nabla v\right\vert -\Lambda_{\infty}v^{Q}%
,-\Delta_{\infty}v\right\}  =0 & \mathrm{in} & \Omega\\
v=0 & \mathrm{on} & \partial\Omega.
\end{array}
\right.
\]

\end{corollary}

Corollary \ref{cor1} extends to $q(p)\not =p$ the convergence result proved
for $q(p)=p$ by Juutinen, Lindqvist and Manfredi in \cite{JLM}. Indeed, it
shows that if $q(p)\not =p$ and $Q=1,$ then any limit function (as
$p\rightarrow\infty$) of the family $\left(  u_{p,q(p)}\right)  $ is a first
eigenfunction of the infinity Laplacian.

We also observe, as consequence of (\ref{exp1}), that Corollary \ref{cor1}
recovers the convergence results obtained by Charro-Peral in \cite{ChaPe} and
Charro-Parini in \cite{ChaPa}.

Our main purpose in this paper is to study the limiting behavior, as
$p\rightarrow\infty,$ of the least energy solutions to the Dirichlet problem
(\ref{dirq}) in the case not yet treated in the literature: $Q=\infty,$ which
we call \textit{hyperdiffusive} case. We can admit faster growth for $\mu_{p}$
by assuming that
\begin{equation}
\Theta:=\lim_{p\rightarrow\infty}(\mu_{p})^{1/q(p)}\in(0,\infty),
\label{Theta}%
\end{equation}
instead of (\ref{Lamb}). Hence, (\ref{scale}) and (\ref{lim}) yield
\begin{equation}
\lim_{p\rightarrow\infty}v_{p,q(p)}=\Theta^{-1}\lim_{p\rightarrow\infty
}u_{p,q(p)}. \label{Theta1}%
\end{equation}
(Note that $\Theta=1$ under (\ref{Lamb}).)

Therefore, our main result, stated in the sequel, handles the hyperdiffusive
case for $u_{p,q(p)}.$

\begin{theorem}
\label{main2}If
\begin{equation}
\lim_{p\rightarrow\infty}\frac{q(p)}{p}=\infty, \label{hyper}%
\end{equation}
the function $u_{\infty}\in C_{0}(\overline{\Omega})\cap W^{1,\infty}(\Omega)$
obtained in Theorem \ref{main1} is a viscosity solution to the problem%
\begin{equation}
\left\{
\begin{array}
[c]{lll}%
-\Delta_{\infty}u=0 & \text{in} & \Omega\setminus M\\
u=0 & \text{on} & \partial\Omega\\
u=1 & \text{in} & M.
\end{array}
\right.  \label{dirM}%
\end{equation}

Moreover,
\begin{equation}
u_{\infty}=\frac{d_{\Omega}}{\left\Vert d_{\Omega}\right\Vert _{\infty}}
\label{eq}%
\end{equation}
if and only if
\begin{equation}
M=M_{\Omega}=\Sigma_{\Omega}, \label{MMS}%
\end{equation}
where%
\[
\Sigma_{\Omega}:=\left\{  x\in\Omega:d_{\Omega}\text{ is not differentiable at
}x\right\}  .
\]

\end{theorem}

We recall that $\Sigma_{\Omega},$ known as the ridge set of $\Omega,$ is
precisely the set of all points of $\Omega$ whose distance to the boundary is
achieved at least two points in $\partial\Omega.$ Hence, $\Sigma_{\Omega}$
contains $M_{\Omega}.$ A consequence of Theorem \ref{main2} is that the
equality (\ref{eq}) is not possible if $\Sigma_{\Omega}$ is larger than
$M_{\Omega},$ as it is the case of a square and other polygons.

We emphasize that the closed set $M$ is given abstractly as the set of maximum
points of $u_{\infty},$ but it enjoys the property of being a subset of
$M_{\Omega}.$ We also observe that $u_{\infty}$ might depend on a particular
subsequence of the family $\left(  u_{p,q(p)}\right)  ,$ the same occurring
with $M.$ \ However, the uniqueness of $u_{\infty}$ is guaranteed whenever
$M=M_{\Omega}$ since the problem%
\begin{equation}
\left\{
\begin{array}
[c]{lll}%
-\Delta_{\infty}u=0 & \text{\textup{in}} & \Omega\setminus M_{\Omega}\\
u=0 & \text{\textup{on}} & \partial\Omega\\
u=1 & \text{on} & M_{\Omega}%
\end{array}
\right.  \label{umo}%
\end{equation}
has a unique viscosity solution according to the comparison principle by
Jensen in \cite{Jen} (note that $\partial(\Omega\setminus M_{\Omega}%
)=\partial\Omega\cup M_{\Omega}$). Thus, if $M=M_{\Omega}$ then $u_{\infty}$
must be the uniform limit of the whole family $\left(  u_{p,q(p)}\right)  $.
The simplest case is when $M_{\Omega}$ is a singleton, and it occurs when
$\Omega$ is a ball or a square, for instance. Actually, $\Sigma_{B_{R}(x_{0}%
)}=\left\{  x_{0}\right\}  $ for a ball $B_{R}(x_{0})$, so that the equalities
in (\ref{MMS}) hold. Thus, one has
\[
u_{\infty}(x)=1-\frac{\left\vert x-x_{0}\right\vert }{R}\text{ \ for all }x\in
B_{R}(x_{0}).
\]

If $M_{\Omega}$ is not a singleton, the determination of $M$ (or even the
verification that $M$ is a proper subset of $M_{\Omega}$) seems to be a
challenging task as it would require a deeper analysis of the extremal Sobolev
functions $u_{p,q(p)}$ under suitable assumptions on $\Omega$ (possibly
involving geometric aspects such as convexity or special symmetries).

In view of (\ref{Theta1}) the following corollary of Theorem \ref{main2} is
imediate and settles the hyperdiffusive case for $v_{p,q(p)}$ under
(\ref{Theta}).

\begin{corollary}
For all $p$ sufficiently large, let $v_{p,q(p)}$ be a positive least energy
solution to (\ref{dirq}), with $\mu_{p}$ satisfying (\ref{Theta}) and $q(p)$
satisfying (\ref{hyper}). Each sequence $v_{p_{n},q(p_{n})},$ with
$p_{n}\rightarrow\infty,$ admits a subsequence that converges uniformly to a
viscosity solution $v_{\infty}\in C_{0}(\overline{\Omega})\cap W^{1,\infty}(\Omega)$
to the Dirichlet problem%
\[
\left\{
\begin{array}
[c]{lll}%
-\Delta_{\infty}v=0 & \text{in} & \Omega\setminus M\\
v=0 & \text{on} & \partial\Omega\\
v=\Theta^{-1} & \text{in} & M,
\end{array}
\right.
\]
where $M$ is a closed subset of $M_{\Omega}.$ Moreover,
\[
v_{\infty}=\frac{d_{\Omega}}{\Theta\left\Vert d_{\Omega}\right\Vert _{\infty}}%
\]
if and only if $M=M_{\Omega}=\Sigma_{\Omega}.$
\end{corollary}

Inspecting carefully \cite{ChaPe} one can verify that, in view of the limit
(\ref{lim}), the results obtained by Charro and Peral described above are also
valid for $\mu_{p}=\lambda_{p,q(p)}$ when $Q=0,$ but under the hypothesis
(\ref{qinf}). More precisely, such results show that the family $(u_{p,q(p)})$
converges uniformly to $u_{\infty}=\Lambda_{\infty}d_{\Omega}$ since
\[
1\leq q(p)<p,\text{ \ }\lim_{p\rightarrow\infty}q(p)=\infty\text{ \ and
\ }\lim_{p\rightarrow\infty}\frac{q(p)}{p}=0.
\]
The uniqueness of $u_{\infty}$ is guaranteed by the fact (see \cite[Theorem
2.1]{Jen}) that $\Lambda_{\infty}d_{\Omega}$ is the only viscosity solution to
the problem%
\[
\left\{
\begin{array}
[c]{lll}%
\min\left\{  \left\vert \nabla u\right\vert -\Lambda_{\infty},-\Delta_{\infty
}u\right\}  =0 & \mathrm{in} & \Omega\\
u=0 & \mathrm{on} & \partial\Omega.
\end{array}
\right.
\]

In order to complement our analysis on the asymptotic behavior of
$\lambda_{p,q(p)}$ and $u_{p,q(p)}$ we prove the following result in which
$Q=0$ and (\ref{qinf}) is not fulfilled.

\begin{proposition}
\label{main3}If
\begin{equation}
r:=\lim_{p\rightarrow\infty}q(p)<\infty, \label{r}%
\end{equation}
then%
\[
\lim_{p\rightarrow\infty}\lambda_{p,q(p)}^{1/p}=\left\Vert d_{\Omega
}\right\Vert _{r}^{-1},
\]
and%
\[
\lim_{p\rightarrow\infty}u_{p,q(p)}=\frac{d_{\Omega}}{\left\Vert d_{\Omega
}\right\Vert _{r}}\text{ \ uniformly in }\overline{\Omega}.
\]

\end{proposition}

Theorem \ref{main1}, Theorem \ref{main2} and Proposition \ref{main3} are
proved in Section \ref{sec3}.

\section{Preliminaries \label{sec2}}

Let us recall the concept of viscosity solution to the partial differential
equation (PDE)%
\begin{equation}
F(u,\nabla u,D^{2}u)=0, \label{F}%
\end{equation}
where $F:\mathbb{R}\times\mathbb{R}^{N}\times\mathbb{S}^{N}$ is a continuous
function, with $\mathbb{S}^{N}$ denoting the set of all real symmetric
matrices of order $N\times N.$

In the sequel, $B(x_{0})$ will denote a ball centered at $x_{0}.$

\begin{definitions}
Let $D\subset\mathbb{R}^{N}$ be a bounded domain and let $u\in C(D).$ We say that:

\begin{enumerate}
\item $u$ is a viscosity supersolution to (\ref{F})\ in $D$ if, for each
$x_{0}\in D$ and each ball $B(x_{0})\subset D$ one has%
\[
F(\phi(x_{0}),\nabla\phi(x_{0}),D^{2}\phi(x_{0}))\geq0
\]
for every $\phi\in C^{2}(B(x_{0}))$ such that
\[
\phi(x)-u(x)<0=\phi(x_{0})-u(x_{0})\text{ \ for all }x\in B(x_{0}%
)\setminus\left\{  x_{0}\right\}  .
\]

\item $u$ is a viscosity subsolution to (\ref{F}) in $D$ if, for each
$x_{0}\in D$ and each ball $B(x_{0})\subset D$ one has%
\[
F(\phi(x_{0}),\nabla\phi(x_{0}),D^{2}\phi(x_{0}))\leq0
\]
for every $\phi\in C^{2}(B(x_{0}))$ such that%
\[
\phi(x)-u(x)>0=\phi(x_{0})-u(x_{0})\text{ \ for all }x\in B(x_{0}%
)\setminus\left\{  x_{0}\right\}  .
\]

\item $u$ is a viscosity solution to (\ref{F}) in $D$ if $u$ is both a
viscosity subsolution and a viscosity supersolution to $F$ in $D.$
\end{enumerate}
\end{definitions}

We say that $u\in C(D)$ is infinity harmonic in $D$ if $u$ is a viscosity
solution to the PDE
\begin{equation}
-\Delta_{\infty}u=0 \label{ilap}%
\end{equation}
in $D.$ If $u$ is a viscosity supersolution (subsolution) to (\ref{ilap}) in
$D$ we say that $u$ is infinity superharmonic (subharmonic) in $D$ and use
$-\Delta_{\infty}u\geq0$ ($-\Delta_{\infty}u\leq0$) in $D$ as\ notation.

The following result is well known (see \cite[Theorem 1]{Bhat} and
\cite[Corollary 4.5]{ML}).

\begin{lemma}
\label{+}Let $D\subset\mathbb{R}^{N}$ be a bounded domain and let $u\in C(D).$
If $u$ is a nonnegative and $-\Delta_{\infty}u\geq0$ in $D,$ then either
$u\equiv0$ in $D$ or $u>0$ in $D.$
\end{lemma}

By a viscosity solution to the Dirichlet problem (\ref{dirM}) we mean a
function $u\in C(\overline{\Omega})$ that is infinity harmonic in
$\Omega\setminus M$ and such that: $u=0$ on $\partial\Omega$ and $u=1$ on $M.$

We observe that
\[
\Delta_{p}u=(p-2)\left\vert \nabla u\right\vert ^{p-4}\left\{  \frac{1}%
{p-2}\left\vert \nabla u\right\vert ^{2}\Delta u+\Delta_{\infty}u\right\}
\]
whenever $u$ is a function of class $C^{2}.$ Thus, in the viscosity solution
approach the partial differential equation in (\ref{Dir}) is usually written
as
\begin{equation}
-(p-2)\left\vert \nabla u\right\vert ^{p-4}\left\{  \frac{1}{p-2}\left\vert
\nabla u\right\vert ^{2}\Delta u+\Delta_{\infty}u\right\}  -\lambda
_{p,q}\left\vert u\right\vert ^{q-2}u=0. \label{Dirv}%
\end{equation}

By a viscosity solution to (\ref{Dir}) we mean a viscosity solution $u\in
C(\overline{\Omega})$ to (\ref{Dirv}) that vanishes on $\partial\Omega.$

We will make use of the following known facts stated in form of lemmas.

\begin{lemma}
\label{F1}Let $p_{n}\rightarrow\infty$ and for each $n\in\mathbb{N}$ let
$u_{n}\in W_{0}^{1,p_{n}}(\Omega)$ be nonnegative in $\Omega.$ Suppose that%
\[
\limsup_{n\rightarrow\infty}\left\Vert \nabla u_{n}\right\Vert _{p_{n}}\leq
C.
\]
There exists a subsequence of $\left(  u_{n}\right)  $ converging uniformly in
$\overline{\Omega}$ to a function $u_{\infty}\in C_{0}(\overline{\Omega})\cap
W^{1,\infty}(\Omega)$ such that
\[
0\leq u_{\infty}(x)\leq\left\Vert \nabla u\right\Vert _{\infty}d_{\Omega
}(x)\leq Cd_{\Omega}(x)\text{ \ for all }x\in\overline{\Omega}.
\]

\end{lemma}

\begin{lemma}
\label{F2}For each $p>N$ the embedding $W_{0}^{1,p}(\Omega)\hookrightarrow
C(\overline{\Omega})$ is compact and the infimum
\[
\Lambda_{p}:=\inf\left\{  \left\Vert \nabla u\right\Vert _{p}:u\in W_{0}%
^{1,p}(\Omega)\text{ \ and \ }\left\Vert u\right\Vert _{\infty}=1\right\}
\]
is reached at a function $u_{p}\in W_{0}^{1,p}(\Omega)\cap C^{0,1-\frac{N}{p}%
}(\overline{\Omega}).$ Moreover (see \cite{EP16}),
\begin{equation}
\lim_{p\rightarrow\infty}\Lambda_{p}=\Lambda_{\infty}. \label{a}%
\end{equation}

\end{lemma}

The proof of the next lemma follows directly from the proof of Lemma 1.8 in
\cite{JLM}, where the case $q(p)=p$ is treated.

\begin{lemma}
\label{pvis}The Sobolev extremal function $u_{p,q(p)}$ is a viscosity solution
to
\[
\left\{
\begin{array}
[c]{lll}%
-\Delta_{p}u=\lambda_{p,q(p)}\left\vert u\right\vert ^{q(p)-2}u &
\text{\textup{in}} & \Omega\\
u=0 & \text{\textup{on}} & \partial\Omega.
\end{array}
\right.
\]

\end{lemma}

\begin{lemma}
\label{lm2}The function $d_{\Omega}/\left\Vert d_{\Omega}\right\Vert _{\infty
}$ is the only viscosity solution to (\ref{umo}) if and only if $\Sigma
_{\Omega}=M_{\Omega}.$
\end{lemma}

\begin{proof}
It is well known (see \cite{Ar}, \cite[Corollary 3.4]{Crand}) that
$-\Delta_{\infty}d_{\Omega}=0$ in $\Omega\setminus\Sigma_{\Omega}$ in the
viscosity sense. Hence, if $\Sigma_{\Omega}=M_{\Omega}$ then $d_{\Omega
}/\left\Vert d_{\Omega}\right\Vert _{\infty}$ is the only viscosity solution
to (\ref{umo}).

Inversely, if $d_{\Omega}/\left\Vert d_{\Omega}\right\Vert _{\infty}$ is the
viscosity solution to (\ref{umo}), then it is differentiable in $\Omega
\setminus M_{\Omega}$ as it is infinity harmonic on this set (see
\cite[Theorem 3.2]{Ev}). It follows that $\Sigma_{\Omega}=M_{\Omega}.$
\end{proof}

In order to simplify the notation in the proofs, in the sequel we will denote
$u_{p_{n},q(p_{n})}$ by $u_{n}$ and $\lambda_{p_{n},q(p_{n})}$ by $\lambda
_{n},$ whenever $p_{n}\rightarrow\infty.$

\section{Proofs \label{sec3}}

\subsection{Proof of Theorem \ref{main1}}

\begin{proof}
As
\[
\Lambda_{p}\leq\frac{\left\Vert \nabla u_{p,q(p)}\right\Vert _{p}}{\left\Vert
u_{p,q(p)}\right\Vert _{\infty}}\text{ \ and \ }1=\left\Vert u_{p,q(p)}%
\right\Vert _{q(p)}\leq\left\Vert u_{p,q(p)}\right\Vert _{\infty}\left\vert
\Omega\right\vert ^{1/q(p)}%
\]
we have that%
\begin{equation}
\lambda_{p,q(p)}^{1/p}=\left\Vert \nabla u_{p,q(p)}\right\Vert _{p}\geq
\Lambda_{p}\left\Vert u_{p,q(p)}\right\Vert _{\infty}\geq\Lambda_{p}\left\vert
\Omega\right\vert ^{-1/q(p)}. \label{b}%
\end{equation}

Hence, combining (\ref{a}) with (\ref{b}) we get, on the one hand,
\begin{equation}
\liminf_{p\rightarrow\infty}\lambda_{p,q(p)}^{1/p}\geq\lim_{p\rightarrow
\infty}\Lambda_{p}\lim_{p\rightarrow\infty}\left\vert \Omega\right\vert
^{-1/q(p)}=\Lambda_{\infty}. \label{d}%
\end{equation}

On the other hand, as
\[
\lambda_{p,q(p)}^{1/p}\leq\frac{\left\Vert \nabla d_{\Omega}\right\Vert _{p}%
}{\left\Vert d_{\Omega}\right\Vert _{q(p)}}=\frac{\left\vert \Omega\right\vert
^{1/p}}{\left\Vert d_{\Omega}\right\Vert _{q(p)}}%
\]
we obtain%
\begin{equation}
\limsup_{p\rightarrow\infty}\lambda_{p,q(p)}^{1/p}\leq\lim_{p\rightarrow
\infty}\frac{\left\vert \Omega\right\vert ^{1/p}}{\left\Vert d_{\Omega
}\right\Vert _{q(p)}}=\frac{1}{\left\Vert d_{\Omega}\right\Vert _{\infty}%
}=\Lambda_{\infty}. \label{c}%
\end{equation}

Gathering (\ref{d}) and (\ref{c}) we conclude the proof of (\ref{lim}). Hence,
returning to (\ref{b}) and taking (\ref{a}) into account again, we get
(\ref{lim1}).

Now, let $p_{n}\rightarrow\infty.$ As $\left\Vert \nabla u_{n}\right\Vert
_{p_{n}}=\lambda_{n}^{1/p_{n}}\rightarrow\Lambda_{\infty}$ it follows from
(\ref{c}) and Lemma \ref{F1} that, up to a subsequence, $u_{n}\rightarrow
u_{\infty}\in C_{0}(\overline{\Omega})\cap W^{1,\infty}(\Omega),$ uniformly in
$\overline{\Omega},$ with
\[
0\leq u_{\infty}(x)\leq\left\Vert \nabla u_{\infty}\right\Vert _{\infty
}d_{\Omega}(x)\leq\Lambda_{\infty}d_{\Omega}(x)\text{ \ for all }x\in
\overline{\Omega}.
\]

The uniform convergence and the limit in (\ref{lim1}) imply that $\left\Vert
u_{\infty}\right\Vert _{\infty}=1.$ Therefore,
\[
\left\Vert \nabla u_{\infty}\right\Vert _{\infty}\leq\Lambda_{\infty}\leq
\frac{\left\Vert \nabla u_{\infty}\right\Vert _{\infty}}{\left\Vert u_{\infty
}\right\Vert _{\infty}}=\left\Vert \nabla u_{\infty}\right\Vert _{\infty},
\]
so that $\left\Vert \nabla u_{\infty}\right\Vert _{\infty}=\Lambda_{\infty},$
completing thus the proof of items 1 and 2 of Theorem \ref{main1}.

The proof of item 3 is direct: if $x_{0}\in M,$ then \
\[
1=u_{\infty}(x_{0})\leq\Lambda_{\infty}d_{\Omega}(x_{0})=\frac{d_{\Omega
}(x_{0})}{\left\Vert d_{\Omega}\right\Vert _{\infty}}\leq1,
\]
so that $x_{0}\in M_{\Omega}.$

In order to prove item 4, let us first verify that $-\Delta_{\infty}u_{\infty
}\geq0$ in $\Omega.$ Let $x_{0}\in\Omega$ and fix a ball $B(x_{0}%
)\subset\Omega$ and a function $\phi\in C^{2}(B(x_{0}))$ such that
\[
\phi(x)-u_{\infty}(x)<0=\phi(x_{0})-u_{\infty}(x_{0})\text{ \ for all }x\in
B(x_{0})\setminus\left\{  x_{0}\right\}  .
\]

If $\left\vert \nabla\phi(x_{0})\right\vert =0,$ then we have trivially that
$-\Delta_{\infty}\phi(x_{0})=0.$ Thus, we can assume that $\left\vert
\nabla\phi(x_{0})\right\vert >0.$

As $u_{n}\rightarrow u_{\infty}$ uniformly, there exists $x_{n}\in\Omega$ such
that $x_{n}\rightarrow x_{0}$ and%
\[
\phi(x)-u_{n}(x)<0=\phi(x_{n})-u_{n}(x_{n})\text{ \ for all }\,x\in
B(x_{n})\setminus\{x_{n}\}.
\]
(See \cite[Lemma 4.5]{Li15}). We can also assume that $\left\vert \nabla
\phi(x_{n})\right\vert >0.$

Owing to Lemma \ref{pvis},%
\[
-\Delta_{p_{n}}\phi(x_{n})\geq\lambda_{n}\phi(x_{n})^{q_{n}-1}=\lambda
_{n}u_{n}(x_{n})^{q_{n}-1}\geq0,
\]
that is,%
\[
-(p_{n}-2)\left\vert \nabla\phi(x_{n})\right\vert ^{p_{n}-4}\left\{  \frac
{1}{p_{n}-2}\left\vert \nabla\phi(x_{n})\right\vert ^{2}\Delta\phi
(x_{n})+\Delta_{\infty}\phi(x_{n})\right\}  \geq0.
\]

It follows that%
\[
-\Delta_{\infty}\phi(x_{n})\geq\frac{1}{p_{n}-2}\left\vert \nabla\phi
(x_{n})\right\vert ^{2}\Delta\phi(x_{n}),
\]
and hence, after letting $n\rightarrow\infty,$ we obtain%
\[
-\Delta_{\infty}\phi(x_{0})\geq\left\vert \nabla\phi(x_{0})\right\vert
^{2}\Delta\phi(x_{0})\lim_{n\rightarrow\infty}\frac{1}{p_{n}-2}=0.
\]

This shows that $-\Delta_{\infty}u_{\infty}\geq0$ in $\Omega$ and in addition
proves that $u_{\infty}>0$ in $\Omega$ according to Lemma \ref{+}.
\end{proof}

\subsection{Proof of Theorem \ref{main2}}

\begin{proof}
According to item 4 of Theorem \ref{main1}, $-\Delta_{\infty}u_{\infty}\geq0$
in $\Omega\setminus M.$ Thus, as $u_{\infty}=0$ on $\partial\Omega$ and
$u_{\infty}=1$ on $M,$ it remains to prove that $-\Delta_{\infty}u_{\infty
}\leq0$ in $\Omega\setminus M.$ At this point we use (\ref{hyper}), a stronger
hypothesis than (\ref{qinf}).

We recall that
\[
M:=\left\{  x\in\Omega:u_{\infty}(x)=1\right\}  \subseteq M_{\Omega}:=\left\{
x\in\Omega:d_{\Omega}(x)=\left\Vert d_{\Omega}\right\Vert _{\infty}\right\}
,
\]
so that $\Omega\setminus M$ is an open set.

Let us fix $x_{0}\in\Omega\setminus M,$ a ball $B(x_{0})\subset\Omega\setminus
M,$ and a function $\phi\in C^{2}(B(x_{0}))$ such that
\[
\phi(x)-u_{\infty}(x)>0=\phi(x_{0})-u_{\infty}(x_{0})\text{ \ for all }x\in
B(x_{0})\setminus\left\{  x_{0}\right\}  .
\]

If $\left\vert \nabla\phi(x_{0})\right\vert =0$ then we obtain directly that
$-\Delta_{\infty}\phi(x_{0})=0.$ Thus, we assume that $\left\vert \nabla
\phi(x_{0})\right\vert >0.$

Let $p_{n}\rightarrow\infty$ be such that $u_{n}\rightarrow u_{\infty}$
uniformly. We recall from (\ref{lim}) and (\ref{hyper})\ that%
\begin{equation}
\lim_{n\rightarrow\infty}\lambda_{n}^{1/p_{n}}=\Lambda_{\infty}\text{ \ and
\ }\lim_{n\rightarrow\infty}\frac{q_{n}}{p_{n}}=\infty. \label{A}%
\end{equation}

The uniform convergence $u_{n}\rightarrow u_{\infty}$ guarantees the existence
of $x_{n}\in\Omega\setminus M$ such that $x_{n}\rightarrow x_{0}$ and%
\[
\phi(x)-u_{n}(x)>0=\phi(x_{n})-u_{n}(x_{n})\text{ \ for all }\,x\in
B(x_{n})\setminus\{x_{n}\}.
\]

As $\left\vert \nabla\phi(x_{0})\right\vert >0$ and $0<u_{\infty}(x_{0})<1$ we
can assume that $\left\vert \nabla\phi(x_{n})\right\vert >0$ and%
\begin{equation}
0<a\leq u_{n}(x_{n})\leq b<1\text{ \ for all }n\in\mathbb{N}, \label{B}%
\end{equation}
where the constants $a$ and $b$ are uniform with respect to $n.$

Using again Lemma \ref{pvis} we have%
\begin{align*}
-(p_{n}-2)\left\vert \nabla\phi(x_{n})\right\vert ^{p_{n}-4}\left\{
\frac{\left\vert \nabla\phi(x_{n})\right\vert ^{2}}{p_{n}-2}\Delta\phi
(x_{n})+\Delta_{\infty}\phi(x_{n})\right\}   &  \leq\lambda_{n}\phi
(x_{n})^{q_{n}-1}\\
&  =\lambda_{n}u_{n}(x_{n})^{q_{n}-1}.
\end{align*}
Hence, after rearranging terms we obtain%
\begin{equation}
-\frac{\left\vert \nabla\phi(x_{n})\right\vert ^{2}}{p_{n}-2}\Delta\phi
(x_{n})-\Delta_{\infty}\phi(x_{n})\leq\frac{1}{p_{n}-2}\left[  \frac
{\lambda_{n}^{\frac{1}{p_{n}-4}}u_{n}(x_{n})^{\frac{q_{n}-1}{p_{n}-4}}%
}{\left\vert \nabla\phi(x_{n})\right\vert }\right]  ^{p_{n}-4}. \label{C}%
\end{equation}

Combining (\ref{A}) and (\ref{B}) we have%
\[
\lim_{n\rightarrow\infty}\frac{\lambda_{n}^{\frac{1}{p_{n}-4}}u_{n}%
(x_{n})^{\frac{q_{n}-1}{p_{n}-4}}}{\left\vert \nabla\phi(x_{n})\right\vert
}=\frac{\Lambda_{\infty}}{\left\vert \nabla\phi(x_{0})\right\vert }%
\lim_{n\rightarrow\infty}u_{n}(x_{n})^{\frac{q_{n}-1}{p_{n}-4}}=0
\]
so that%
\[
\lim_{n\rightarrow\infty}\frac{1}{p_{n}-2}\left[  \frac{\lambda_{n}^{\frac
{1}{p_{n}-4}}u_{n}(x_{n})^{\frac{q_{n}-1}{p_{n}-4}}}{\left\vert \nabla
\phi(x_{n})\right\vert }\right]  ^{p_{n}-4}=0.
\]

As
\[
\lim_{n\rightarrow\infty}\left\{  -\frac{1}{p_{n}-2}\left\vert \nabla
\phi(x_{n})\right\vert ^{2}\Delta\phi(x_{n})-\Delta_{\infty}\phi
(x_{n})\right\}  =-\Delta_{\infty}\phi(x_{0}),
\]
we conclude, after letting $n\rightarrow\infty$ in (\ref{C}), that%
\[
-\Delta_{\infty}\phi(x_{0})\leq0.
\]
This shows that $u_{\infty}$ is subharmonic in $\Omega\setminus M$ and
finishes the proof that $u_{\infty}$ is harmonic in $\Omega\setminus M.$

The first equality in (\ref{MMS}) implies that $u_{\infty}$ is a viscosity
solution to (\ref{umo}) and by Lemma \ref{lm2} the second equality
in\ (\ref{MMS}) implies that $d_{\Omega}/\left\Vert d_{\Omega}\right\Vert
_{\infty}$ is also a viscosity solution to the same problem. Thus, by
uniqueness $u_{\infty}=d_{\Omega}/\left\Vert d_{\Omega}\right\Vert _{\infty}.$

Inversely, if $u_{\infty}=d_{\Omega}/\left\Vert d_{\Omega}\right\Vert
_{\infty}$ in $\Omega\setminus M,$ then $d_{\Omega}$ is harmonic and,
therefore, differentiable in $\Omega\setminus M.$ Consequently, $M=\Sigma
_{\Omega}$ and this implies that $\Sigma_{\Omega}=M_{\Omega}$ (as $M\subseteq
M_{\Omega}\subseteq\Sigma_{\Omega}$).
\end{proof}

\subsection{Proof of Proposition \ref{main3}}

\begin{proof}
The simplest case in which the function $q(p)$ is constant, say $q(p)\equiv
s\in\lbrack1,\infty),$ has already been proved in \cite[Theorem 4.2]{EPM22}.
According to that result
\begin{equation}
\lim_{p\rightarrow\infty}\lambda_{p,s}^{1/p}=\left\Vert d_{\Omega}\right\Vert
_{s}^{-1} \label{s1}%
\end{equation}
and%
\begin{equation}
\lim_{p\rightarrow\infty}u_{p,s}=\frac{d_{\Omega}}{\left\Vert d_{\Omega
}\right\Vert _{s}}\text{ \ uniformly in }\overline{\Omega}. \label{s2}%
\end{equation}

In the sequel we combine (\ref{s1})-(\ref{s2}) with the following fact: the
function $q\mapsto\left\vert \Omega\right\vert ^{\frac{p}{q}}\lambda_{p,q}$ is
strictly decreasing, for each fixed $p\in\lbrack1,\infty).$ This monotonicity
result was proved for $1<p<N$ in \cite[Proposition 2]{E13}, but the proof
given there also works for $p>N$ (see also \cite{An}).

Let $\epsilon>0$ be arbitrarily fixed. Owing to (\ref{r}) we have that
$q(p)<r+\epsilon$ for all $p$ sufficiently large, say $p>p_{0}.$ Hence,
according to the above mentioned monotonicity result we have%
\[
\lambda_{p,r+\epsilon}\left\vert \Omega\right\vert ^{\frac{p}{r+\epsilon
}-\frac{p}{q(p)}}\leq\lambda_{p,q(p)}\text{ \ for all }p>p_{0},
\]
so that%
\[
\frac{\left\vert \Omega\right\vert ^{\frac{1}{r+\epsilon}-\frac{1}{r}}%
}{\left\Vert d_{\Omega}\right\Vert _{r+\epsilon}}\leq\liminf_{p\rightarrow
\infty}\lambda_{p,q(p)}^{1/p}.
\]

Letting $\epsilon\rightarrow0,$ we obtain on the one hand%
\begin{equation}
\left\Vert d_{\Omega}\right\Vert _{r}^{-1}\leq\liminf_{p\rightarrow\infty
}\lambda_{p,q(p)}^{1/p}. \label{s3}%
\end{equation}

On the other hand, the inequality%
\[
\lambda_{p,q(p)}^{1/p}\leq\frac{\left\Vert \nabla d_{\Omega}\right\Vert _{p}%
}{\left\Vert d_{\Omega}\right\Vert _{q(p)}}=\frac{\left\vert \Omega\right\vert
^{1/p}}{\left\Vert d_{\Omega}\right\Vert _{q(p)}}%
\]
lead us to the estimate
\begin{equation}
\limsup_{p\rightarrow\infty}\lambda_{p,q(p)}^{1/p}\leq\left\Vert d_{\Omega
}\right\Vert _{r}^{-1}. \label{s0}%
\end{equation}

Combining (\ref{s3}) and (\ref{s0}) we conclude that%
\[
\lim_{p\rightarrow\infty}\lambda_{p,q(p)}^{1/p}=\left\Vert d_{\Omega
}\right\Vert _{r}^{-1}.
\]

Now, let $p_{n}\rightarrow\infty.$ As $\lambda_{p,q(p)}^{1/p}=\left\Vert
\nabla u_{p,q(p)}\right\Vert _{p},$ Lemma \ref{F1} and (\ref{s0}) imply that,
up to a subsequence, $u_{p_{n},q(p_{n})}$ converges uniformly in
$\overline{\Omega}$ to a function $u_{\infty}\in C_{0}(\overline{\Omega})\cap
W^{1,\infty}(\Omega)$ satisfying%
\begin{equation}
0\leq u_{\infty}\leq\frac{d_{\Omega}}{\left\Vert d_{\Omega}\right\Vert _{r}%
}\text{ in }\overline{\Omega}. \label{s4}%
\end{equation}

The uniform convergence and (\ref{r}) yield%
\[
\lim_{n\rightarrow\infty}\left\Vert u_{p_{n},q(p_{n})}\right\Vert _{r}%
=\lim_{n\rightarrow\infty}\left\Vert u_{p_{n},q(p_{n})}\right\Vert _{q(p_{n}%
)}=\left\Vert u_{\infty}\right\Vert _{r}.
\]
Hence, recalling that $\left\Vert u_{p_{n},q(p_{n})}\right\Vert _{q(p_{n}%
)}=1,$ we have that $\left\Vert u_{\infty}\right\Vert _{r}=1,$ which combined
with (\ref{s4}) yields $u_{\infty}\equiv\frac{d_{\Omega}}{\left\Vert
d_{\Omega}\right\Vert _{r}}.$

As the limit function $u_{\infty}$ is always the same, we conclude that
\[
\lim_{p\rightarrow\infty}u_{p,q(p)}=\frac{d_{\Omega}}{\left\Vert d_{\Omega
}\right\Vert _{r}}\text{ \ uniformly in }\overline{\Omega}.
\]

\end{proof}

\section*{Acknowledgements}

The author was partially supported by Fapemig/Brazil (RED-00133-21) and
CNPq/Brazil (305578/2020-0).

\end{document}